\documentclass[12pt]{amsart}
\usepackage{fullpage,url,amssymb,enumerate,colonequals}
\usepackage{pgfplots}

\usepackage{mathrsfs} % for \mathscr (script letters)
\usepackage{MnSymbol}
\usepackage{extarrows}
\usepackage{lscape}
\usepackage[OT2,T1]{fontenc}
\usepackage{color}

\usepackage[
        colorlinks, citecolor=darkgreen,
        backref,
        pdfauthor={Alain Kraus}, % add other authors
]{hyperref}

\usepackage[french]{babel}

\newcommand{\PP}{\mathbb{P}}
\newcommand{\Q}{\mathbb{Q}}

\newcommand{\sym}{\mathfrak{S}}

\newcommand{\calC}{\mathcal{C}}

\DeclareMathOperator{\Gal}{Gal}

\numberwithin{equation}{section}
\newtheorem{theorem}{Th\'eor\`eme}

\newtheorem{corollary}{Corollaire}
\newtheorem{proposition}{Proposition}

\theoremstyle{definition}

\newtheorem{question}[equation]{Question}

\theoremstyle{remark}
\newtheorem{remark}[equation]{Remarque}

\definecolor{darkgreen}{rgb}{0,0.5,0}

\begin{document}

\title{Points totalement r\'eels de la courbe $x^5+y^5+z^5=0$ }
%\subjclass[2010]{Primary 11G30; Secondary 14G25, 14G40, 14K15, 14K20}

\author{Alain Kraus}

\begin{abstract}  Soient  $\overline \Q$ une cl\^oture alg\'ebrique de $\Q$ et $\Q^{tr}$ le sous-corps de $\overline \Q$ form\'e de la r\'eunion des corps de nombres totalement r\'eels. Pour tout nombre premier $p\geq 3$, soit $F_p/\Q$ la courbe de Fermat d'\'equation $x^p+y^p+z^p=0$.  En 1996, Pop a d\'emontr\'e  que  le corps $\Q^{tr}$ est large. En particulier, l'ensemble $F_p(\Q^{tr})$ des points de $F_p$ rationnels sur $\Q^{tr}$ est infini.  Comment expliciter des points non triviaux ($xyz\neq 0$) de $F_p(\Q^{tr})$ ?
Si on a $p\geq 5$, il semble que   les seuls points d\'ej\`a connus de $F_p(\Q^{tr})$ soient ceux de $F_p(\Q)$ et ils sont triviaux. 
Dans cet article,  on s'int\'eresse  \`a cette question dans le  cas  o\`u $p=5$.   Il n'existe pas  de corps totalement r\'eels de degr\'e sur $\Q$ au plus $5$ sur lesquels $F_5$ a des points non triviaux.
 On se propose ici  d'expliciter  une infinit\'e de points  de $F_5$   rationnels sur des corps totalement r\'eels de degr\'e $6$ sur $\Q$.

  \end{abstract}
\bigskip

\date{\today}

\address{Sorbonne Universit\'e,
Institut de Math\'ematiques de Jussieu - Paris Rive Gauche,
UMR 7586 CNRS - Paris Diderot,
4 Place Jussieu, 75005 Paris, 
France}
\email{alain.kraus@imj-prg.fr}

\renewcommand{\keywordsname}{Mots-cl\'es}

\makeatletter
\@namedef{subjclassname@1991}{MSC 2010}
\makeatother

 \maketitle

\renewcommand{\abstractname}{Abstract}
\begin{abstract}   Let $\overline \Q$ be an algebraic closure of $\Q$ and $\Q^{tr}$  be the subfield of $\overline \Q$ obtained by taking the union of all totally real number fields.
For any prime $p\geq 3$, let $F_p/\Q$ be the Fermat curve of equation $x^p+y^p+z^p=0$. In 1996, Pop has shown that the field $\Q^{tr}$ is large. In particular, the set $F_p(\Q^{tr})$  of the points of $F_p$  rational over $\Q^{tr}$ is infinite. How to explicit non-trivial points $(xyz\neq 0$) in $F_p(\Q^{tr})$ ? If one has $p\geq 5$, it seems that the only  points   already known in  $F_p(\Q^{tr})$ are those of $F_p(\Q)$ and they are trivial. In this paper, we investigate this question in case $p=5$. There  are no totally real fields whose degree over $\Q$ is at most $5$ over which $F_5$ has non-trivial points. We propose here to explicit infinitely many points of $F_5$  rational over totally real fields of degree $6$ over $\Q$. 
\end{abstract}

\section{Introduction}   
Soit $p\geq 3$ un nombre premier. Notons $F_p/\Q$ la courbe de Fermat d'\'equation
$$x^p+y^p+z^p=0.$$ 
Un point $[x,y,z]$ de $F_p(\overline \Q)$ est dit non trivial si $xyz\neq 0$. 

Notons  $\Q^{tr}$ la r\'eunion  des corps de nombres totalement r\'eels dans $\overline \Q$. Wiles a \'etabli en 1994 que $F_p(\Q)$ est r\'eduit aux points triviaux (\cite{Wiles}). Depuis, 
il a \'et\'e d\'emontr\'e, 
pour des familles infinies de  corps  $K$ totalement r\'eels,  que  $F_p(K)$ est r\'eduit aux points triviaux 
si  $p$ est assez grand fonction de $K$  (voir par exemple \cite{FKS}). 
Cela \'etant, 
le corps $\Q^{tr}$ est large (\cite{Pop}, page 2) et $F_p(\Q)$ n'est pas vide. Par suite,   l'ensemble  $F_p(\Q^{tr})$ est infini. Le degr\'e  d'un point de $F_p$ \'etant  le degr\'e sur $\Q$ de son corps de d\'efinition, 
cela sugg\`ere la question suivante:

\begin{question} \label{Q:que1.1} 
Quel est le plus petit entier $d$ tel que $F_p(\Q^{tr})$ contienne des points non triviaux de  degr\'e $d$ \  ?
\end{question} 

%\begin{question} Quel est le plus petit entier $n$ tel que $F_p(\Q^{tr})$ contienne une infinit\'e de points de degr\'e $n$ sur $\Q$ ?
%\end{question} 

Si $p=3$, on a $d=2$. En effet, $F_3$ est une courbe elliptique, $F_3(\Q)$ est r\'eduit aux points triviaux et par exemple 
$[18+17\sqrt{2},18-17\sqrt{2},-42]$ est un point de $F_3$ rationnel sur $\Q(\sqrt{2})$. 
En utilisant~\cite{MAGMA}, on v\'erifie  que  c'est  un point  d'ordre infini de $F_3$.

\`A ma connaissance, si on a   $p\geq 5$, aucun point non trivial de $F_p(\Q^{tr})$  n'a d\'ej\`a  \'et\'e explicit\'e dans la litt\'erature. On d\'emontre dans cet article  que pour $p=5$ on a  $d=6$,  et on explicite une infinit\'e de points totalement r\'eels de degr\'e $6$ sur $F_5$. 
Si on a $p\geq 7$, la question \ref{Q:que1.1}  semble    ouverte. Pour $p=7$, signalons qu'en utilisant les travaux de Sall et Tzermias dans \cite{Sall} et \cite{Tzermias}, on peut \'etablir avec  des arguments analogues \`a ceux \'evoqu\'es   ci-dessous  que l'on a $d\geq 10$.

Ce qui pr\'ec\`ede sugg\`ere aussi le  probl\`eme suivant. Comment d\'emontrer qu'il existe un entier $n$ tel que $F_p(\Q^{tr})$ contienne une infinit\'e de points de degr\'e $n$, si tel est le cas ?
Pour $p=3$ (resp. $p=5$) un tel entier existe   et le plus petit d'entre eux est $n=2$ (resp. $n=6$).  
Si on a $p\geq 7$, ce probl\`eme semble non r\'esolu.

Supposons d\'esormais $p=5$.  Soit $\zeta_3$ une racine primitive cubique de l'unit\'e. D'apr\`es les travaux de  Gross et   Rohrlich,  les seuls points quadratiques de $F_5$ sont     (\cite{GR}, Theorem 5.1)
 \begin{equation}
  \label{(1.1)}
 P=[\zeta_3,\zeta_3^2,1]\quad \text{et}\quad \overline P=[\zeta_3^2,\zeta_3,1].
 \end{equation}
 Par ailleurs,  Klassen et Tzermias ont d\'emontr\'e qu'il n'existe pas de points cubiques sur  $F_5$,
 et que 
 les points de $F_5$ de degr\'e 
 $4$ ou $5$   s'obtiennent comme l'intersection de $F_5$ avec une droite d\'efinie sur $\Q$  (\cite{KlassenTzermias}, Theorem 1).

On en d\'eduit que l'on a $d\geq 6$ i.e. qu'il n'existe  pas de corps 
 totalement r\'eels,   de degr\'e   au plus  $5$ sur $\Q$, sur lesquels $F_5$ a des points non triviaux. En effet, 
 supposons qu'il existe un point non trivial $A=[x,y,1]\in F_5(\Q^{tr})$  de degr\'e au plus  $5$. D'apr\`es les r\'esultats rappel\'es ci-dessus, le degr\'e de $A$ est $4$ ou $5$. Quitte \`a permuter $x$ et $y$, il existe donc des nombres rationnels $a$ et $b$ tels que $y=ax+b$. En posant $F=X^5+(aX+b)^5+1\in \Q[X]$, on a ainsi 
 $F(x)=0$ et $\Q(x)$ est le corps de d\'efinition de $A$. Or  on v\'erifie directement que $F$ poss\`ede  au plus trois racines r\'eelles, d'o\`u une contradiction et notre assertion. 

Klassen et Tzermias  ont aussi d\'ecrit g\'eom\'etriquement   les points de $F_5$ de degr\'e $6$.
Ils \'etablissent que 
ces points  s'obtiennent comme l'intersection de $F_5$ avec quatre familles de  coniques planes sur $\Q$ (\cite{KlassenTzermias}, Theorem 1).
On d\'ecrit dans le paragraphe \ref{3}
la famille des coniques planes  sur $\Q$, irr\'eductibles sur $\Q$, passant par  $P$ et  ayant comme  tangente en  $P$ celle de $F_5$ en $P$. 
En d\'eterminant l'intersection  de ces coniques avec $F_5$, 
on d\'emontre   qu'il existe une infinit\'e de corps 
totalement r\'eels, galoisiens sur $\Q$ de degr\'e $6$, de groupe de Galois   isomorphe \`a $\sym_3$,   sur lesquels $F_5$    a un point non trivial.

Tous les calculs num\'eriques que ce travail a n\'ecessit\'es ont \'et\'e effectu\'es \`a l'aide des logiciels de calculs {\tt Pari-gp} (\cite{Pari}) et {\tt Magma} (\cite{MAGMA}).
Il se trouve dans  \cite{FK}, un fichier  {\tt Magma} qui a \'et\'e \'ecrit par Nuno Freitas, ainsi qu'un fichier  {\tt Pari-gp}, 
permettant de v\'erifier  ces calculs.

{\bf{Remerciements.}} Je remercie vivement Nicolas Billerey et Nuno Freitas pour   les remarques  dont ils m'ont  fait part au cours de  ce travail,  ainsi que pour l'aide informatique qu'ils m'ont apport\'ee. Je remercie \'egalement Dominique Bernardi qui a r\'ealis\'e la figure intervenant dans l'exemple du paragraphe \ref{2}, ainsi 
que 
le rapporteur de cet article pour  tous les commentaires  tr\`es instructifs qu'il m'a communiqu\'es et qui ont am\'elior\'e la premi\`ere version 
de ce travail.

\section{\'Enonc\'e des r\'esultats \label{2}}
Dans toute la suite,  la lettre  $t$  d\'esigne un nombre rationnel distinct de $2$. Posons 

$$u=\frac{3t^2 - 2t + 2}{t^2 + t - 1},\quad v=\frac{t^5 - 5t^4 + 10t^3 - 20t^2 + 15t - 7}{(t-2)(t^2+t-1)^2},$$

$$w=\frac{-3t^5 + 10t^4 - 20t^3 + 20t^2 - 20t + 6}{(t-2)(t^2+t-1)^2}.$$

Notons $f_t$ le polyn\^ome de $\Q[X]$ d\'efini par l'\'egalit\'e
\begin{equation}
 \label{(2.1)}
 f_t=X^6+uX^5+vX^4+wX^3+vX^2+uX+1.
\end{equation}

Posons par ailleurs 
$$s=(t^4 - 3t^3 - t^2 + 3t + 1)(t-2) \quad \text{(on a }\ s\neq 0),$$
$$a_0=-\frac{(t^2 + 1)(t^3 - t^2 + 2t - 3)}{s},
\quad a_1=-\frac{3t^7 - 9t^6 + 16t^5 - 15t^4 + 10t^3 - 11t^2 + 8t - 7}{(t^2 + t - 1)s},$$

$$a_2=\frac{2t^8 - 14t^7 + 52t^6 - 99t^5 + 100t^4 - 54t^3 + 38t^2 - 44t + 13}{(t^2 + t - 1)(t-2)s},$$

$$a_3=\frac{t^8 + t^7 - 21t^6 + 65t^5 - 90t^4 + 78t^3 - 57t^2 + 32t - 15}{(t^2+t-1)(t-2)s},$$

$$a_4=-\frac{2t^5 - 6t^4 + 13t^3 - 14t^2 + 7t - 5}{s},\quad 
a_5=-\frac{(t^2 + t - 1)(t^3 - t^2 + 2t - 3)}{s}.$$

D\'esignons par $K_t$ le corps de d\'ecomposition de $f_t$ dans $\overline {\Q}$. Soit $\alpha\in K_t$ une racine de $f_t$. Posons 
\begin{equation}
 \label{(2.2)}
\beta=\sum_{i=0}^{5} a_i\alpha^i%\quad \text{et}\quad \gamma=\sum_{i=0}^{5} a_i(1/\alpha)^i.
\end{equation}

\begin{theorem}    \label{T:thm1}  Supposons $t\neq 1$. 

\begin{itemize}

\item[1)]  Le polyn\^ome $f_t\in \Q[X]$ est irr\'eductible sur $\Q$. 

\item[2)] L'ensemble des six racines de $f_t$ dans $K_t$ est 
$$\big\lbrace \alpha, \beta,\beta/\alpha,1/\alpha,1/\beta,\alpha/\beta\big\rbrace.$$
 En particulier, on a $K_t=\Q(\alpha)$ et 
 l'extension $\Q(\alpha)/\Q$ est galoisienne  de degr\'e $6$. Son  groupe de Galois est isomorphe \`a $\sym_3$.

\item[3)]   Les points
$$[\alpha,\beta,1],\quad [1/\alpha,\beta/\alpha, 1],\quad[1/\beta,\alpha/\beta,1],$$
et ceux obtenus en permutant leurs deux premi\`eres coordonn\'ees,  
appartiennent \`a  $F_5(K_t)$. Ils  sont distincts et non triviaux. L'ensemble de ces  points est  l'orbite galoisienne et la $\sym_3$-orbite de chacun d'eux.
 
\end{itemize} 
\end{theorem} 

\begin{remark} \label{R:rmq2.3}  
Les six points de $F_5(K_t)$ d\'ecrits ci-dessus  forment l'ensemble des points de degr\'e $6$  dans 
 l'intersection de $F_5$ avec la  conique $C_t$ d'\'equation \eqref{(3.5)}. Le groupe $\sym_3$ agit  sur $F_5$ et  $C_t$  par permutation des coordonn\'ees. Il en r\'esulte  que l'ensemble de ces points est la $\sym_3$-orbite de chacun d'eux.
%Par ailleurs, $\sym_3$ agit sur $F_5$ par permutation des coordonn\'ees. 
%L'ensemble de ces points est  la $\sym_3$-orbite de chacun d'eux.  
%Cela provient du fait qu'ils s'obtiennent comme l'intersection de $F_5$ avec une conique sur laquelle $\sym_3$ agit   par permutation des coordonn\'ees 
\end{remark} 
 
\begin{theorem}   \label{T:thm2} 
Soit $r$ le nombre r\'eel  tel que  $7r^5 - 10r^4 - 20r^3 - 4=0$. On a   $r\simeq 2,558$.    
 
\begin{itemize}
\item[1)]
Le corps $K_t$ est totalement r\'eel si et seulement si on a $2<t<r$. 
\item[2)]  Il existe une infinit\'e de nombres rationnels $t$  tels que $2<t<r$ et que les corps $K_t$ soient deux \`a   deux distincts.
\end{itemize} 
\end{theorem} 
On en d\'eduit l'\'enonc\'e suivant:
\begin{corollary} \label{C:cor1} Il  existe une infinit\'e de corps de nombres $K$
totalement r\'eels, galoisiens sur $\Q$ de degr\'e $6$, de groupe de Galois sur $\Q$  isomorphe \`a $\sym_3$,  tels que $F_5(K) $  poss\`ede un point non trivial.
\end{corollary}

 {\bf{Exemple.}} Prenons $t=5/2$. On a 
 $$f_t=X^6 + \frac{63}{31}X^5 - \frac{1149}{961}X^4 - \frac{4283}{961}X^3 - \frac{1149}{961}X^2 + \frac{63}{31}X + 1.$$
Le corps $K_t=\Q(\alpha)$ est totalement r\'eel et on a 
 $$\alpha^5+\beta^5+1=0\quad \text{avec}\quad \beta=\frac{2821}{89}\alpha^5+\frac{2850}{89}\alpha^4-\frac{196815}{2759}\alpha^3-\frac{188718}{2759}\alpha^2+\frac{90989}{2759}\alpha+\frac{2639}{89}.$$

On  constate sur la figure ci-dessous  que l'intersection de $F_5$ avec la conique  $C_t$
 est form\'ee de six points r\'eels, qui constituent l'orbite galoisienne et la $\sym_3$-orbite de $[\alpha,\beta,1]$. 
\bigskip

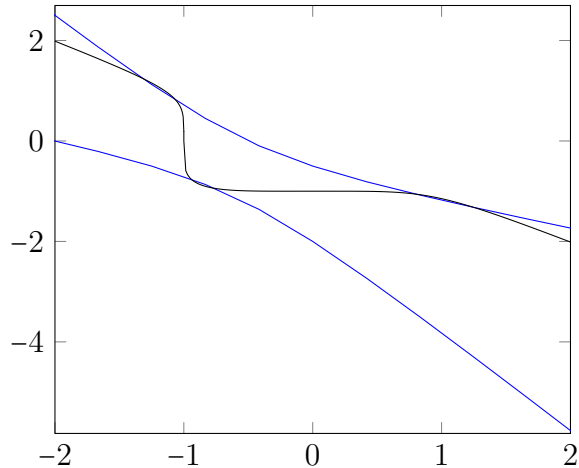
\begin{center}
\pgfplotsset{compat=1.17}
\begin{tikzpicture}
    \begin{axis}[no markers, xmin=-2,xmax=2]
    %  use TeX as calculator:
    \addplot {(sqrt(9*x^2 +10*x +9)-5*(x+1))/4};
    \addplot [color=blue]{(-sqrt(9*x^2 +10*x +9)-5*(x+1))/4};
    \addplot [domain=-1:2,samples=200]{-(1+x^5)^0.2};
    \addplot [domain=-2:-1,samples=200]{abs(x^5+1)^0.2};
    \draw (-1,0.2)--(-1,0);
      \end{axis}
\end{tikzpicture}
\end{center}
\medskip

\begin{remark} Les trois autres  familles de coniques 
d\'ecrites dans le th\'eor\`eme 1 de \cite{KlassenTzermias} forment une orbite sous l'action de $\sym_3$. On peut d\'emontrer que leur intersection avec $F_5$   ne contient pas de points totalement r\'eels de degr\'e $6$. On obtient ainsi, avec les th\'eor\`emes \ref{T:thm1} et \ref{T:thm2}, une description de tous les points totalement r\'eels de degr\'e $6$ de $F_5$. 
%On peut d\'emontrer  que l'intersection de $F_5$ avec les trois autres  familles des coniques 
%d\'ecrites dans le th\'eor\`eme 1 de \cite{KlassenTzermias}, qui forment une orbite sous l'action de $\sym_3$, n'est pas form\'ee de points totalement r\'eels de degr\'e $6$. 
\end{remark}

 \section{La conique $C_t/\Q$  \label{3}}
Rappelons que les points $P$ et $\overline P$  sont d\'efinis par les \'egalit\'es~\eqref{(1.1)}.   D\'ecrivons  la famille des coniques projectives planes sur $\Q$,  irr\'eductibles sur $\Q$, passant par  $P$ et ayant comme  tangente en  $P$   celle de $F_5$ en $P$. 

Soit $\calC$ une conique projective  plane d\'efinie sur $\Q$. Il existe $a,b,c,d,e,f$ dans $\Q$ tels que  $\calC$ poss\`ede une \'equation de   la forme
$$ax^2+by^2+cz^2+dxy+exz+fyz=0.$$

\begin{proposition}   \label{P:prop1}   1) Supposons que $\calC$ soit irr\'eductible sur $\Q$ et que $P$ appartienne à $\calC$. Alors, $P$ est lisse. 

2) Les deux conditions suivantes sont \'equivalentes:

\begin{itemize}
\item[2.1)] La conique $\calC$ est  irr\'eductible sur $\Q$, le point $P$ appartient \`a $\calC$ et la tangente \`a $\calC$ en $P$ est celle  de $F_5$ en $P$.
\item[2.2)] On a 
\begin{equation}
\label{(3.1)} 
a=b=c, \quad d=e=f\quad \text{et}\quad d\neq 2a.
\end{equation} 
\end{itemize} 
\end{proposition}

\begin{proof}  
Notons $F$ le polyn\^ome homog\`ene de degr\'e $2$ d\'efinissant  $\calC$ et $F_x$, $F_y$, $F_z$ ses polyn\^omes dériv\'es par rapport \`a $x$, $y$ et $z$.

1)  La conique  $\calC$ \'etant  d\'efinie sur $\Q$,  le point $P$ appartient \`a $\calC$ et est lisse si et seulement si tel est le cas de $\overline P$. Si $P$ \'etait singulier, $\calC$ serait donc la droite double sur $\Q$ passant par $P$ et $\overline P$, contredisant ainsi notre hypoth\`ese. V\'erifions ce fait directement. 
On a les \'egalit\'es
\begin{equation}
\label{(3.2)} 
F_x(P)=2a\zeta_3+d\zeta_3^2+e,\quad  F_y(P)=2b\zeta_3^2+d\zeta_3+f,\quad F_z(P)=2c+e\zeta_3+f\zeta_3^2.
\end{equation}
Supposons que l'on ait $F_x(P)=F_y(P)=F_z(P)=0$. 
En utilisant l'\'egalit\'e,   $\zeta_3^2=-1-\zeta_3$, on obtient les conditions $a=b=c$, $e=d=f$ et $d=2a$, d'o\`u  $F=a(x+y+z)^2$ et  l'assertion.

2) Supposons que la  condition 2.1 soit satisfaite. D'apr\`es l'assertion pr\'ec\'edente, 
 l'\'equation de la tangente \`a  $\calC$ en $P$ est 
$$F_x(P)x+F_y(P)y+F_z(P)z=0.$$
L'\'equation de la tangente   \`a $F_5$  en $P$ est 
$$\zeta_3x+\zeta_3^2y+z=0.$$
D'apr\`es  l'hypoth\`ese faite,
il existe donc $\lambda \in \overline \Q^*$ tel que 
$$\lambda (\zeta_3,\zeta_3^2,1)=(F_x(P),F_y(P),F_z(P)).$$
On obtient  $\lambda=F_z(P)$, d'o\`u 
$$\zeta_3 F_z(P)=F_x(P)\quad \text{et}\quad \zeta_3^2F_z(P) =F_y(P).$$
On en d\'eduit  avec \eqref{(3.2)} que l'on a  
$$-2a+d-e+2c=0\quad \text{et} \quad d-2e+f=0,$$
$$-d-2c+f+2b=0\quad \text{et} \quad  e-2c-f + 2b=0.$$
La diff\'erence entre les deux derni\`eres \'egalit\'es implique la relation $e-2f +d=0$. 
Avec l'\'egalit\'e $d-2e+f=0$, on obtient alors $e=f=d$, puis $a=b=c$.
De plus, on a $d\neq 2a$, sinon $F=a(x+y+z)^2$, ce qui n'est pas, d'o\`u la  condition \eqref{(3.1)}.

Inversement, supposons  que la condition 2.2 soit satisfaite. On v\'erifie que $\calC$ est  irr\'eductible sur $\overline \Q$  si et seulement si  on a $(2a-d)(a+d)\neq 0$ i.e. $a+d\neq 0$ (cf. \cite{Walker}, Chapter III,  Theorem 6.1).
Par suite, 
si $a+d\neq 0$, alors $\calC$ est en particulier irr\'eductible sur $\Q$. Si $a+d=0$,  l'\'equation de $\calC$ est $x^2+y^2+z^2-(xy+xz+yz)=0$, et on constate   avec \cite{Pari} que $\calC$ 
est irr\'eductible sur $\Q$.  Par ailleurs, on a 
$d\neq 2a$,  donc  $P$ est un point lisse de $\calC$ et  la tangente \`a  $\calC$ en $P$ est celle de $F_5$ en $P$, d'o\`u la condition 2.1.
 \end{proof} 
 
 \begin{remark} Le pinceau des  coniques d\'efinies par la condition~\eqref{(3.1)} est engendr\'e par la droite double $(x+y+z)^2$ passant par $P$ et $\overline P$, et le produit 
 $$(\zeta_3x+\zeta_3^2y+z)(\zeta_3^2+\zeta_3y+z)=x^2+y^2+z^2-(xy+xz+yz),$$
 des \'equations des tangentes  en  $P$ et $\overline P$. 
 \end{remark} 
 
 \begin{remark}
Les points $P$ et $\overline P$ \'etant conjugu\'es sur $\Q$, si la condition 2.1 est satisfaite, alors $\overline P$ est un point lisse de  $\calC$  et la tangente \`a $\calC$ en $\overline P$ est celle  de $F_5$ en $\overline P$.  
\end{remark}

La d\'etermination de l'intersection 
de $F_5$ avec  la famille de coniques  v\'erifiant la condition  \eqref{(3.1)}  fournit ainsi, g\'en\'eriquement, des points de $F_5$ rationnels sur des corps de degr\'e $6$ sur $\Q$ (\cite{KlassenTzermias}, Theorem 1). Afin de d\'emontrer les r\'esultats que l'on a en vue, on se limitera au cas o\`u $a\neq 0$, l'\'equation 
de ces coniques \'etant alors  de la forme $x^2+y^2+z^2+t(xy+xz+yz)=0$ avec $t\in \Q$ et $t\neq 2$. En fait, on  constate avec la d\'emonstration 
du th\'eor\`eme \ref{T:thm1} que, si  $t$ est distinct de $1$, 
ces coniques ont  avec $F_5$ un  contact d'ordre $2$ en $P$  et $\overline P$. 
%En fait, en consid\'erant l'intersection de $F_5$ avec la conique d'\'equation $xy+xz+yz=0$, on obtient des points de $F_5$ rationnels sur un corps de degr\'e $6$ qui n'est %pas totalement r\'eel. 

 Pour tout nombre rationnel $t\neq 2$, notons   d\'esormais $C_t$ la conique d\'efinie sur $\Q$ d'\'equation 
\begin{equation}
 \label{(3.5)}
x^2+y^2+z^2+t(xy+xz+yz)=0.
\end{equation}

 \section{L'intersection $F_5\cap C_t$  \label{4}} 
 On v\'erifie que l'intersection de  la droite d'\'equation $z=0$ avec $F_5\cap C_t$ est vide. D\'ecrivons $F_5\cap C_t$ dans l'ouvert $z=1$. 
 
 Rappelons que, pour $t\neq 2$,  le  polyn\^ome $f_t\in \Q[X]$  est  d\'efini par l'\'egalit\'e \eqref{(2.1)} et que les nombres rationnels $a_i$ (fonctions de $t$) ont \'et\'e d\'efinis dans le 
 paragraphe~\ref{2}.   La proposition  qui suit n'est pas indispensable pour \'etablir nos r\'esultats, mais elle permet de comprendre comment l'\'enonc\'e du th\'eor\`eme~\ref{T:thm1} a \'et\'e trouv\'e. On utilisera  essentiellement  dans la suite la proposition~\ref{P:prop3} ci-dessous.
 
 \begin{proposition}   \label{P:prop2} Soit $[x,y,1]$ un point de $F_5\cap C_t$, distinct de $P$ et $\overline P$.  
On a 
\begin{equation}
 \label{(4.1)}
 f_t(x)=0\quad \text{et}\quad y=\sum_{i=0}^{5}  a_i x^i.
 \end{equation}
 \end{proposition} 

 \begin{proof}  Compte tenu de l'\'equation \eqref{(3.5)}, consid\'erons  le r\'esultant  $R_t\in \Q[X]$ par rapport \`a $Y$ des  polyn\^omes de $\Q[X,Y]$
$$X^5+Y^5+1\quad \text{et}\quad X^2+Y^2+1+t(XY+X+Y).$$
 On a  l'\'egalit\'e (cf. \cite{Pari})
$$R_t=(2-t)(t^2+t-1)^2(X^2+X+1)^2f_t.$$
Le point  $[x,y,1]$ \'etant   distinct de $P$ et  $\overline P$,  $x$ n'est pas $\zeta_3$ ni  $\zeta_3^2$. On en d\'eduit que l'on a
 $$f_t(x)=0.$$
 Par ailleurs, on a  $$y^2=-1-x^2-t(xy+x+y).$$
  Cette \'egalit\'e permet d'exprimer $y^5$ comme un polyn\^ome de degr\'e $1$ en $y$, dont les coefficients d\'ependent de $x$ et $t$. En utilisant les relations
  $$x^5+y^5+1=0\quad \text{et}\quad f_t(x)=0,$$
  on constate  alors  que   $y$ v\'erifie la seconde \'egalit\'e de \eqref{(4.1)} (voir \cite{FK}), 
d'o\`u le r\'esultat. 
\end{proof} 

 \begin{proposition}  \label{P:prop3} Soit $x$ un \'el\'ement de $\overline \Q$ tel que $f_t(x)=0$. Posons 
 $$y=\sum_{i=0}^5 a_ix^i.$$
 Alors, on a $f_t(y)=0$ et le point $[x,y,1]$ appartient \`a $F_5\cap C_t$. 
 \end{proposition}
  
  \begin{proof}  On v\'erifie  directement cet \'enonc\'e en utilisant  \cite{FK}.
  \end{proof} 
 
 \begin{remark} \label{R:rmq4.2}  D\'ecrivons g\'eom\'etriquement les points  de  $F_5\cap C_t$ distincts de $P$ et $\overline P$.  On dispose du morphisme $F_5\to F_5/\sym_3$ qui \`a tout point de $F_5$ associe sa $\sym_3$-orbite. 
% Il est de degr\'e $6$ et on d\'eduit de  la formule de Riemann-Hurwitz que la courbe $F_5/\sym_3$ est  de genre $0$.  
Par ailleurs, l'application $\varphi : F_5/\sym_3\to \PP^1$  d\'efinie dans un ouvert convenable par l'\'egalit\'e 
 $$\varphi(\text{orbite de}\ [x,y,z])=[t,1]\quad \text{avec}\quad t=-\frac{x^2+y^2+z^2}{xy+xz+yz},$$
 se prolonge  en un morphisme de degr\'e $1$ de $F_5/\sym_3$ sur $\PP^1$, qui est 
donc un 
 isomorphisme. En particulier, si $t$ est un nombre rationnel distinct de $1$ et $2$, 
 la fibre en $[t,1]$ du morphisme $F_5\to F_5/\sym_3\simeq \PP^1$ ainsi obtenu, est $F_5\cap C_t$ priv\'e de $P$ et $\overline P$, et c'est la $\sym_3$-orbite de chacun de ses points. 
 \end{remark}

 \section{D\'emonstration du Th\'eor\`eme~\ref{T:thm1}}
\subsection{D\'emonstration de  l'assertion 1}
Supposons que $f_t$ soit divisible par un polyn\^ome unitaire $g\in \Q[X]$, irr\'eductible sur $\Q$,  de degr\'e $1$, $2$ ou  $3$. Soit $x\in \overline \Q$ une racine de $g$.
On a en particulier $f_t(x)=0$. Posons 
$$y=\sum_{i=0}^5 a_i x^i.$$
Le point $[x,y,1]$ appartient \`a $F_5$ (prop.~\ref{P:prop3}). Son corps de d\'efinition  est  $\Q(x)$. 

Il n'existe pas de points cubiques sur $F_5$ (\cite {KlassenTzermias}, Theorem 1). Par suite, le degr\'e de $g$ est $1$ ou $2$. Si $g$ est degr\'e $1$, vu que $F_5(\Q)$ est r\'eduit aux points triviaux, on a  $xy=0$. On a  $f_t(x)=f_t(y)=0$ (prop.~\ref{P:prop3}) or $f_t(0)=1$, d'o\`u une contradiction. 
Ainsi, $g$ est de degr\'e $2$.  Il en r\'esulte  que  $[x,y,1]$ est $P$ ou $\overline P$,  et donc que $x$ est  une  racine primitive cubique de l'unit\'e  (\cite{GR}, Theorem 5.1). On en d\'eduit  que l'on a 
$g=X^2+X+1$.  
Le reste de la division euclidienne de $f_t$ par $X^2+X+1$ est 
$$\frac{5(t^4 - 3t^3 + 4t^2 - 2t + 1)(t-1)}{(2-t)(t^2+t-1)^2}.$$
On obtient  $t=1$, ce qui par  l'hypoth\`ese  est exclu, d'o\`u une contradiction et le r\'esultat. 

\begin{remark}  \label{R:rmq5.1} Pour $t=1$, on a $f_t=(X^2+X+1)^3$.
\end{remark}

\subsection{D\'emonstration des  assertions 2 et 3} %Par d\'efinition, on a $f_t(\alpha)=0$. 
%Le polyn\^ome $f_t$ est r\'eciproque,   donc
%$1/\alpha$ est aussi racine de $f_t$. 
%Les \'el\'ements $\beta$ et $\gamma$ \'etant d\'efinis par les \'egalit\'es~\eqref{(2.2)}, 
%ce sont  des racines de $f_t$ (prop.~\ref{P:prop3}) et il en est de m\^eme 
%de $1/\beta$ et $1/\gamma$. 
On v\'erifie   avec~\cite{FK} que l'on a l'\'egalit\'e
$$f_t=(X-\alpha)(X-\beta)(X-\beta/\alpha)(X-1/\alpha)(X-1/\beta)(X-\alpha/\beta).$$ 
Le discriminant de $f_t$ est non nul, car on a $t\neq 1$, donc $f_t$ est s\'eparable. 
On obtient ainsi l'ensemble annonc\'e des  racines de $f_t$. En particulier, on a $K_t=\Q(\alpha)$ et l'extension $\Q(\alpha)/\Q$ est galoisienne de degr\'e $6$. %(L'\'egalit\'e suffit \`a \'etablir directement ces assertion.)

D\'emontrons  que le groupe de Galois de $K_t/\Q$ est isomorphe \`a $\sym_3$. 
Il existe $\sigma_1$ et $\sigma_2$ dans  $\Gal(K_t/\Q)$ tels que l'on ait
$$\sigma_1(\alpha)=1/\alpha\quad \text{et}\quad  \sigma_2(\beta)=1/\beta.$$
Les \'el\'ements $\sigma_1$ et $\sigma_2$ sont d'ordre $2$. V\'erifions qu'ils sont distincts, ce qui prouvera que $ \Gal(K_t/\Q)$ n'est pas cyclique, donc est isomorphe \`a $\sym_3$. D'apr\`es l'\'egalit\'e \eqref{(2.2)},  on a 
$$\sigma_1(\beta)=\sum_{i=0}^5 a_i \sigma_1(\alpha)^i=\sum_{i=0}^5 a_i(1/\alpha)^i.$$
Par ailleurs, on a   
$$\beta/\alpha=\sum_{i=0}^5 a_i(1/\alpha)^i\quad \text{et}\quad \beta/\alpha\neq 1/\beta.$$
On en d\'eduit que $\sigma_1\neq \sigma_2$, d'o\`u    la seconde assertion du th\'eor\`eme. 

V\'erifions la troisi\`eme assertion. On a $f_t(\alpha)=0$, donc  le point $[\alpha,\beta,1]$ appartient \`a $F_5(K_t)$ (prop. \ref{P:prop3}).
On d\'eduit alors de la premi\`ere assertion  que son orbite galoisienne est de cardinal $6$,  et  qu'elle est form\'ee des points d\'ecrits dans l'\'enonc\'e  
du th\'eor\`eme.  Compte tenu de la remarque~\ref{R:rmq2.3},  cela \'etablit   le r\'esultat. 

%$[1/\alpha,\beta/\alpha, 1]=[1,\beta,\alpha]$ et $[1/\beta,\alpha/\beta,1]=[1,\alpha,\beta]$. 
% En permutant les deux premi\`eres coordonn\'ees de ces trois points, on obtient ainsi 
% six points distincts et non triviaux de $F_5(K_t)$. On d\'eduit alors de la premi\`ere assertion du th\'eor\`eme que l'ensemble de ces points est l'orbite galoisienne de chacun %d'eux. Ce sont les six points de degr\'e $6$ de $F_5\cap C_t$.
 %Ces trois points ainsi obtenus  sont distincts et non triviaux, d'o\`u le r\'esultat.
% 1/\alpha$ et $1/\beta$ sont des racines de $f_t$. De plus, on a 
%$$\alpha/\beta=\sum_{i=0}^5a_i (1/\beta)^i.$$
%D'apr\`es  ce qui pr\'ec\`ede et la proposition~\ref{P:prop3},   les points 
%$[\alpha,\beta,1]$,  $[1/\alpha,\beta/\alpha, 1]$ et $[1/\beta,\alpha/\beta,1]$ sont donc dans $F_5(K_t)$.

 \section{D\'emonstration du Th\'eor\`eme~\ref{T:thm2}}

\subsection{D\'emonstration de  l'assertion 1}  D'apr\`es la remarque~\ref{R:rmq5.1}, on a $K_1=\Q(\zeta_3)$. On peut donc supposer $t\neq 1$. 
 L'extension $K_t/\Q$ \'etant galoisienne de groupe de Galois isomorphe \`a $\sym_3$ (th. \ref{T:thm1}), 
  $K_t$ contient   trois corps cubiques  non galoisiens sur $\Q$. Parce que  $f_t$ est un polyn\^ome r\'eciproque, l'un d'entre eux est  
$\Q(\xi)$ o\`u $\xi=\alpha+1/\alpha$. Le polyn\^ome minimal de $\xi$ sur $\Q$ est 
$$g_t=X^3 + uX^2 + (v - 3)X  -2u + w.$$
Le discriminant $\Delta$  de $g_t$ est 
$$\Delta=-\frac{5^2(t^4 - 3t^3 - t^2 + 3t + 1)^2(7t^5 - 10t^4 - 20t^3 - 4)(t-1)^2}{(t-2)^3(t^2+t-1)^{6}}.$$
On constate que  $\Delta$ modulo $\Q^{*2}$ est $(2-t)(7t^5 - 10t^4 - 20t^3 - 4)$. Le corps  $K_t$ est donc  le compos\'e  de $\Q(\xi)$ 
et du corps quadratique 
$$\Q\left(\sqrt{(2-t)(7t^5 - 10t^4 - 20t^3 - 4)}\right).$$
 Le polyn\^ome  $g_t$  a trois racines r\'eelles si et seulement si on a  $\Delta>0$. 
 Il en r\'esulte que $K_t$ est totalement r\'eel si et seulement si on a 
$$(2-t)(7t^5 - 10t^4 - 20t^3 - 4)>0.$$
On v\'erifie directement que cette condition signifie que l'on a $2<t<r$. 

\begin{remark} D'apr\`es le th\'eor\`eme~\ref{T:thm1} et la d\'emonstration de cette assertion,  l'ensemble des points rationnels sur $\Q$ de la courbe  de genre $2$ d'\'equation $y^2=(2-x)(7x^5 - 10x^4 - 20x^3 - 4)$ est le singleton $\lbrace (2,0)\rbrace$. On peut aussi  constater  avec \cite{MAGMA} que le groupe de Mordell-Weil sur $\Q$ de la Jacobienne de cette courbe est trivial.
%On peut aussi le constater avec {\tt Magma}  (\cite{MAGMA}).
\end{remark}

\subsection{D\'emonstration de  l'assertion 2} Supposons qu'il n'existe qu'un nombre fini de rationnels $t\in ]2,r[$ tels que les corps $K_t$ soient deux \`a deux distincts. 
Dans ce cas, il existe un rationnel $t_0\in ]2,r[$  et une infinit\'e de $t\in ]2,r[\cap \ \Q$ tels que  l'on ait  $K_{t_0}=K_t$.  Indiquons deux arguments conduisant \`a une contradiction.

Soient $t$ et $t'$  deux nombres rationnels distincts dans $]2,r[$. 
On d\'eduit de la remarque \ref{R:rmq4.2} que  les $\sym_3$-orbites des points  de degr\'e $6$ de $F_5\cap C_t$ et $F_5\cap C_{t'}$ sont distinctes. 
Ces points sont rationnels sur $K_t$ et $K_{t'}$. 
Il en r\'esulte que $F_5(K_{t_0})$ est infini, ce qui contredit  le th\'eor\`eme  de Faltings \cite{Faltings} Satz 7.

On peut aussi proc\'eder comme suit. Soit $d_0$ l'entier sans facteurs carr\'es tel que $\Q(\sqrt{d_0})$ soit le corps quadratique contenu dans $K_{t_0}$. On a constat\'e  que 
$\Q\left(\sqrt{(2-t)(7t^5 - 10t^4 - 20t^3 - 4)}\right)$ est le corps quadratique contenu dans $K_t$. 
La courbe de genre $2$ d'\'equation
$$d_0y^2=(2-x)(7x^5 - 10x^4 - 20x^3 - 4)$$
poss\`ede donc une infinit\'e de points rationnels sur $\Q$, d'o\`u la contradiction cherch\'ee.

\end{document}